\documentclass[10pt, a4paper]{article}

\usepackage{amssymb, amsmath, amsthm, mathrsfs, color}
\usepackage{mathbbol, graphicx, latexsym, tkz-graph,mathtools}
\usepackage[nobysame]{amsrefs}

\usepackage[margin = 3 cm]{geometry}
\usepackage{array}

\newcolumntype{C}[1]{>{\centering\arraybackslash$}p{#1}<{$}}

\newtheorem{theorem}{Theorem}[section]
\newtheorem{lemma}[theorem]{Lemma}
\newtheorem{corollary}[theorem]{Corollary}
\newtheorem{definition}{Definition}[section]
\theoremstyle{remark}
\newtheorem{remark}[theorem]{Remark}
\newtheorem{example}[theorem]{Example}

\begin{document}

\title{Simultaneous diagonalization of conics in $PG(2,q)$}

\author{Katharina Kusejko (ETH Z\"urich)}

\date{}

\maketitle

\begin{abstract}
Consider two symmetric $3 \times 3$ matrices $A$ and $B$ with entries in $GF(q)$, for $q=p^n$, $p$ an odd prime. The zero sets of $v^T Av$ and $v^T Bv$, for $v \in GF(q)^3$ and $v\neq 0$, can be viewed as (possibly degenerate) conics in $PG(2,q)$, the Desarguesian plane of order $q$. Using combinatorial properties of pencils of conics in $PG(2,q)$, we are able to tell when it is possible to find a regular matrix $S$ with entries in $GF(q)$, such that $S^T A S$ and $S^T BS$ are both diagonal matrices. This is equivalent to the existence of a collineation, which maps two given conics into two conics in diagonal form. For two proper conics, we will in particular compare the situation in $PG(2,q)$ to the real projective plane and compare the geometrical properties of being diagonalizable with our combinatorial results. 
\end{abstract}



\section{Introduction}
\label{intro}

A well-known problem in linear algebra is that of finding for two given Hermitian (symmetric) $3 \times 3$ matrices $A$ and $B$ over the complex (real) numbers a matrix $S$, such that $S^*AS$ and $S^*BS$ are both in diagonal form, where $S^*$ is the conjugate transpose of $S$ (see for example \cite{MR0094690} or \cite{MR818902}). If such a matrix $S$ can be found, $A$ and $B$ are said to be simultaneously diagonalizable. In this paper, we show that the situation over finite fields is for some cases rather different to the situation over the real or complex numbers. For example, if $A$ and $B$ are two real symmetric matrices, a sufficient but not necessary condition for $A$ and $B$ being simultaneously diagonalizable is that $v^TAv$ and $v^TBv$ have no non-trivial common zeros, where $v$ is a column vector. When viewing $v^TAv$ and $v^TBv$ as polynomials which define conics in the real projective plane, the condition above would be equivalent with the conics being disjoint. We will see that for finite projective Desarguesian planes, not all pairs of disjoint conics are simultaneously diagonalizable. In particular, we obtain the following condition for two proper disjoint conics in $PG(2,q)$ being simultaneously diagonalizable:

\begin{theorem} Two proper disjoint conics $C$ and $\tilde{C}$ in $PG(2,q)$ are simultaneously diagonalizable if and only if their pencil $\mathcal{P}(C,\tilde{C})$ leads to a partition of the form $\{P,g,C_1,...,C_{q-1}\}$ or $\{P,\tilde{P},g\tilde{g},C_1,...,C_{q-2}\}$, where $P$, $\tilde{P}$ are points, $g$ is a line, $g\tilde{g}$ a pair of lines and $C_i$ proper conics.
\end{theorem}

We are therefore interested in geometric conditions, such that the pencil of two proper disjoint conics is of the form $\{ P,g,C_1,...,C_{q-1} \}$ or $\{ P,\tilde{P},g\tilde{g},C_1,...,C_{q-2} \}$. The main result concerning this question is the following theorem about conics $C$ and $\tilde{C}$ in nested position, i.e.\ all points of $C$ are external points of $\tilde{C}$ or all points of $C$ are internal points of $\tilde{C}$ and vice versa.

\begin{theorem} Consider two proper disjoint conics $C$ and $\tilde{C}$ in $PG(2,q)$. The pencil $\mathcal{P}(C,\tilde{C})$ is of the form $\{P,g,C_1,...,C_{q-1}\}$ if and only if all pairs of proper conics lie in nested position.
\end{theorem}

For two proper conics which do intersect in exactly one or exactly three points, we will see that simultaneous diagonalization is never possible. Two proper conics which intersect in exactly four points, however, can always be diagonalized simultaneously. For two proper conics which intersect in exactly two points, the property of being simultaneously diagonalizable depends again on the form of their pencil.

We start by reminding the reader of the most important properties of finite projective planes and their collineations. Moreover, we make some combinatorial considerations used later on in this paper. After that, we will describe a partition of the plane $PG(2,q)$ by considering the pencil of conics defined by two given conics and think about its possible shapes. In the main part, we will then use these pencils to answer the question about simultaneous diagonalization of two given conics. Moreover, we mention some connections to geometrical properties of simultaneously diagonalizable conics. In addition, we will compare our findings to the situation in the real projective plane.

\section{Preliminaries}

\subsection{The projective plane $PG(2,q)$}

We recall briefly the most important facts and notations about finite fields (see \cite{MR1429394}) and finite projective planes (see \cite{MR1612570}).

Let $GF(q)$ be the Galois field of order $q=p^n$ for $p$ an odd prime, i.e.\
$$ GF(q) = \{0,1,\alpha,\alpha^2,\alpha^3,...,\alpha^{q-2}\},$$
for a primitive element $\alpha$ of $GF(q)$.

The finite projective Desarguesian plane constructed over $GF(q)$ is denoted by $PG(2,q)$. The points of $PG(2,q)$ are given by non-zero column vectors $[x,y,z]^T$ for $x,y,z \in GF(q)$, where $[\lambda x,\lambda y,\lambda z] = [x,y,z]$ for all $\lambda \in GF(q)\setminus \{0\}$. Similarly, all lines are denoted by row vectors $[x,y,z]$.

The main idea for the simultaneous diagonalization of two symmetric matrices $A$ and $B$ with entries in $GF(q)$ is to consider the zeros of $v^T A v$ and $v^T B v$, for $v$ a point in $PG(2,q)$. Such a zero set corresponds to a line, a point, a pair of lines or a proper conic in $PG(2,q)$. 

All points, lines, pairs of lines and proper conics in $PG(2,q)$ can be described as sets of the form
\begin{equation} \label{QuForm} V:=\mathbb{V}(F(x,y,z))=\left\{ [x,y,z]^T \in PG(2,q) \big \vert  F(x,y,z)=0 \right\}
\end{equation}
for $ F(x,y,z)=ax^2+by^2+cz^2+dxy+exz+fyz$ with $a,b,c,d,e,f \in GF(q)$, where at least one of the coefficients is different from zero. We call the sets defined by (\ref{QuForm}) conics, and talk about proper conics when referring to sets of $q+1$ points, no three of them collinear. If the set $V$ corresponds to a point, a line or a pair of lines in $PG(2,q)$, we refer to these sets as degenerate conics.

Another  way to look at  (\ref{QuForm}) is to consider the matrix representation 
\begin{equation} \label{MatrixRep}
v^T A v
\end{equation}
for $v=[x,y,z]^T$ and
\begin{equation} \label{Cmatrix}
A= \left(\begin{array}{*3{C{1.5em}}} 2a & d & e \\ d & 2b & f \\ e & f & 2c \end{array}\right).
\end{equation}

Then the set $V$ defined by (\ref{QuForm}) corresponds to a proper conic if and only if the
corresponding matrix $A$ is  regular. If  $A$
corresponds  to a  singular matrix  and (\ref{MatrixRep}) is irreducible, the set $V=\mathbb{V}(v^TAv)$
is one  point  only.  Otherwise,  if (\ref{MatrixRep}) splits into  one or two linear factors, it  corresponds to one or  two lines. Conics whose matrix representation is singular are called {\em degenerate} or {\em non-proper} conics.

If $C$ is  a given proper conic, a line $l$ intersects $C$ in at most two points. 
We call $l$ a {\em tangent}, if it intersects $C$ in exactly one point, 
a  {\em secant}, if it intersects $C$ in exactly two points and  an {\em external line}  if it misses $C$. 
In finite projective planes of odd order, there is either none or exactly two tangents of $C$ through a point off $C$. 
A point $P$ is called
{\em internal point} if there is no tangent to $C$ through $P$ and {\em
  external point} if  there are two tangents to $C$ through $P$. 
  Recall that $(AP)^T$ gives the polar line of $P$ with respect to $C$, 
  where $A$ is the matrix representation (\ref{Cmatrix}) of the conic $C$. 
  In particular, $(AP)^T$ is a tangent if and only if $P$ is on $C$, 
  it is a secant if and only if $P$ is an external point and an external line otherwise.

An important tool are collinear maps of $PG(2,q)$:
\begin{lemma} If $S$ is a regular $3\times 3$ matrix with coefficients
  in $GF(q)$, then $\phi_S:  PG(2,q) \rightarrow PG(2,q),\ P \mapsto
  SP$ is bijective and collinear, i.e.\, a set of collinear points is mapped to a set
  of collinear points.
\end{lemma}

The most important result concerning collineations we are going to use is known as the fundamental theorem of projective geometry, see \cite{MR1612570}*{Section 2.1}.

\begin{lemma} \label{quadrilateral}
Let $\left\{P_1,P_2,P_3,P_4\right\}$ and $\left\{Q_1,Q_2,Q_3,Q_4\right\}$ be sets of points in $PG(2,q)$, such that no three points taken from the same set are collinear. Then there exists a collineation $\phi_S$ such that $\phi_S(P_i)=Q_i$ for $i=1,2,3,4$.
\end{lemma}
 
\subsection{Some combinatorial considerations}

In this section, we count the number of proper conics, which intersect in a particular number of points and use this to count pairs of proper conics which have a certain number of points in common. The numbers obtained in this section will be used later on to count the number of pencils of conics in a particular shape. First, we deduce how many proper conics there are through a given number of points.

\begin{lemma} \label{CountPoints} In $PG(2,q)$, there are $(q-2)$ proper conics through four given points, $(q-1)^2$ proper conics through three given points, $q^2(q-1)$ proper conics through two given points and $q^2(q^2-1)$ proper conics through one given point. Furthermore, there are $q^5-q^2$ proper conics in total.
\end{lemma}
\begin{proof} To count the number of proper conics through four given points, we use Lemma \ref{quadrilateral} and suppose, without loss of generality, that $[1,0,0]^T$, $[0,1,0]^T$, $[0,0,1]^T$ and $[1,1,1]^T$ lie on the proper conics. It follows that the matrices representing the conics are, up to scaling, of the form
$$ A= \left(\begin{array}{*3{C{3em}}} 0 & 1 & e \\ 1 & 0 & -1-e \\ e & -1-e & 0  \end{array}\right).$$
We need $A$ to be regular, which leads to the conditions $e \neq 0$ and $ e \neq -1$. This gives exactly $q-2$ choices for $e$. 
For the second statement, let $[1,0,0]^T$, $[0,1,0]^T$ and $[0,0,1]^T$ be on the proper conics. Again,
$$ A = \left(\begin{array}{*3{C{1.5em}}} 0 & 1 & e \\ 1 & 0 & f \\ e & f & 0  \end{array}\right)$$
needs to be regular, i.e.\ $e f\neq 0$, which gives $(q-1)^2$ choices. 
Similarly can be proceeded for the number of conics through two or one fixed points, as well as for the number of proper conics in total.
\end{proof} 

Now we can use Lemma \ref{CountPoints} to compute the following quantity:

\begin{definition} The number of proper conics in $PG(2,q)$ which intersect any given proper conic $\tilde{C}$ in exactly $k$ points is denoted by
$$ N_k(q) := \# \left\{ C \ |\text{ C a proper conic, } |C\cap\tilde{C}| = k \text{ for $\tilde{C}$ fixed} \right\}. $$
\end{definition}

\begin{lemma}\label{NrIntersection} In $PG(2,q)$, for $q$ odd, we have:
\begin{align*}
N_5(q) &= 1\\
N_4(q) &= \binom{q+1}{4} (q-3)\\
N_3(q) &= \binom{q+1}{3} ((q-1)^2-1)-4 N_4(q)\\
N_2(q) &= \binom{q+1}{2} (q^2(q-1)-1)-6N_4(q)-3N_3(q) \\
N_1(q) &= (q+1)(q^2(q^2-1)-1)-4N_4(q)-3N_3(q)-2N_2(q) \\
N_0(q) &= (q^5-q^2)-N_5(q)-N_4(q)-N_3(q)-N_2(q)-N_1(q)
\end{align*}
\end{lemma}
\begin{proof}
Since any proper conic is uniquely determined by five of its points, we get $N_5(q)=1$. Let us consider the proper conic $\tilde{C}$ and fix four points on $\tilde{C}$. There are $q-3$ proper conics, which intersect $\tilde{C}$ in exactly these four points, since the total number of proper conics through four points is given by $q-2$, as seen in Lemma \ref{CountPoints}. This is true for any set of four points on $\tilde{C}$, which gives us $N_4(q)=\binom{q+1}{4} (q-3)$.

For the third statement, with the same argument as before, we get $\binom{q+1}{3}((q-1)^2-1)$ proper conics through any three points of the fixed proper conic $\tilde{C}$, not counting $\tilde{C}$ itself. But now, we also count those proper conics intersecting $\tilde{C}$ in four points, we even count these conics $\binom{4}{3}=4$ times. Hence, we have to subtract $4N_4(q)$ to get the desired result. 

Similarly for the fourth statement, recall that there are $q^2(q-1)-1$ proper conics through any two points on $\tilde{C}$, not counting $\tilde{C}$ itself. Again, we also count those proper conics intersecting $\tilde{C}$ in exactly four points and now, we even count them $\binom{4}{2}=6$ times. In addition, we count those proper conics intersecting $\tilde{C}$ in exactly three points, namely $\binom{3}{2}=3$ times. By subtracting these expressions, we obtain the claimed number. 

To obtain $N_1(q)$, a similar discussion can be made. Finally, to find the number of proper conics which are disjoint to $\tilde{C}$, we just have to subtract the number of all proper conics which intersect $\tilde{C}$ in exactly one, two, three or four points, as well as $\tilde{C}$ itself, of the number of all proper conics, which is $q^5-q^2$.
\end{proof}

\section{Partition of $PG(2,q)$}

In this section, we consider again homogeneous polynomials of degree two with coefficients $a_i$, $b_i$, $c_i$, $d_i$, $e_i$, $f_i$ in $GF(q)$, not all zero, i.e.\
\begin{equation} \label{EquE}
E_i= a_i x^2 + b_i y^2 + c_i z^2 + d_i x y+e_i x z+f_i y z
\end{equation} 
and the structure of conics $V_i:=\mathbb{V}(E_i)$ obtained by considering the pencil of such objects over $GF(q)$. For this, let us define a pencil of conics:

\begin{definition}
Let $V_i$ and $V_j$ be two conics given by the zero sets of polynomials $E_i$ and $E_j$ in $PG(2,q)$, respectively. Let $\alpha$ be any primitive element of $GF(q)$. Then we define the pencil of $V_i$ and $V_j$ as
$$ \mathcal{P}(V_i,V_j):= \left\{V_i\right\} \cup \left\{V_j\right\} \cup_{0\leq k \leq q-2} \{\mathbb{V}(E_j+\alpha^k E_i)\} $$
i.e.\ the pencil consists of the zero sets of all polynomials obtained by $GF(q)$-linear combinations of $E_i$ and $E_j$.
\end{definition}
The main goal in this section is to show that starting with two disjoint conics $V_i$ and $V_j$, their pencil leads to a partition of the plane $PG(2,q)$. Note that we do not assume that the starting conics are proper, i.e.\ $V_i$ and $V_j$ can correspond to points, lines and pairs of lines as well as to proper conics.

\begin{remark} \label{PlusVar}
As a direct consequence of the definition, we mention that if a point $P$ lies in two conics $V_i$ and $V_j$, it lies in every element of their pencil $\mathcal{P}(V_i,V_j)$. 
\end{remark}

Since a pencil of two conics is constructed by $GF(q)$-linear combinations of two equations, the following statement is immediate as well.

\begin{lemma} \label{ClosednessAddition} The pencil $\mathcal{P}(V_i,V_j)$ is independent of the representatives $V_i$ and $V_j$.
\end{lemma}

The next result deals with the question, whether all points of the plane $PG(2,q)$ are in some element of $\mathcal{P}(V_i,V_j)$. 

\begin{lemma} \label{partition1} Consider the pencil $\mathcal{P}(V_i,V_j)$. Let $P$ be any point in $PG(2,q)$. If $P$ lies in some $W \in \mathcal{P}(V_i,V_j)$, then it is either only in $W$ or in every element of $\mathcal{P}(V_i,V_j)$. Moreover, every point $P$ of the plane $PG(2,q)$ is contained in at least one element of $\mathcal{P}(V_i,V_j)$.
\end{lemma}
\begin{proof} The first statement is immediate by Lemma \ref{PlusVar}. For the second statement, let $P$ be an arbitrary point in $PG(2,q)$ and assume that $P$ does not lie in every element of the pencil $\mathcal{P}(V_i,V_j)$. By the first statement, $P$ can lie in at most one element of $\mathcal{P}(V_i,V_j)$. Suppose $P$ neither lies in $V_i$ nor in $V_j$, so $E_i(P) \neq 0$ and $E_j(P) \neq 0$. But then, $P$ lies in the element of $\mathcal{P}(V_i,V_j)$ given by the equation $E_j - \frac{E_j(P)}{E_i(P)}E_i$.
\end{proof}

\begin{corollary} \label{TableDisjoint} Let $V_1$ and $V_2$ be disjoint conics. Then the pencil $\mathcal{P}(V_i,V_j)$ gives a partition of all points in $PG(2,q)$.
\end{corollary}
\begin{proof} Since $V_1$ and $V_2$ are disjoint, no point of $PG(2,q)$ can be in more than one element in $\mathcal{P}(V_i,V_j)$ since otherwise, by Lemma  \ref{partition1}, such a point would lie in every conic of $\mathcal{P}(V_i,V_j)$ and hence would be a common point of $V_1$ and $V_2$ as well. Moreover, every point of $PG(2,q)$ is contained in at least one element of $\mathcal{P}(V_i,V_j)$ by Lemma \ref{partition1} again.
\end{proof}

\section{Simultaneous diagonalization}

Recall that we are interested in conics given by the zero set of polynomials $E$ of degree two,
\begin{equation} \label{EquE2}
E =  a x^2 + b y^2 + c z^2 + d x y+e x z+f y z
\end{equation} 
where $a,b,c,d_,e,f \in GF(q)$.

For these polynomials, let us first clarify what we mean by diagonal:

\begin{definition} We call a polynomial (\ref{EquE2}) \emph{diagonal}, if the coefficients of the mixed terms $xy$, $xz$ and $yz$ are zero. We call the polynomial \emph{diagonalizable} if there is a collineation which maps the zeros of (\ref{EquE2}) to the zeros of a diagonal polynomial.
\end{definition}

Let $C$ be a proper conic with matrix representation $A$ and $S$ any 
regular matrix. For any point $P$ on $C$, $SP$ lies on $\tilde{C}$, which is given by the matrix representation ${(S^{-1})}^{T}A S^{-1} $, since
$$ P^T A P = 0 \Leftrightarrow (SP)^T {(S^{-1})}^{T}A S^{-1} SP = 0.$$
Therefore, a symmetric matrix $A$ is called diagonalizable, if there exists a regular matrix $S$, such that $S^T A S$ has diagonal form, which means that all entries off the diagonal are zero. In this chapter we are interested in the existence of a regular matrix $S$ such that for two given matrices $A$ and $B$, the matrices $S^T A S$ and $S^T B S$ are both in diagonal form.

\begin{remark} Note that this definition of diagonalizable differs from the definition one might expect from linear algebra, since we are interested in finding a non-singular matrix $S$ with $S^TAS$ in diagonal form, but $S$ is not necessarily orthogonal, i.e.\ in general, we have $S^T \neq S^{-1}$. What we are doing is to think about the existence of a collineation which takes two conics simultaneously into standard form.
\end{remark}

\subsection{The disjoint case}

We have seen in Corollary \ref{TableDisjoint} that a pencil given by two disjoint conics leads to a partition of the plane. We are now interested in the possible shapes of these partitions. We know that any conic given by the zero set of a polynomial (\ref{EquE2}) corresponds either to a proper conic, a point, a line or a pair of lines. Recall that any two distinct lines always intersect in a unique point, hence when starting with two disjoint conics, at most one line or one pair of lines can occur in the whole partition. Since every point has to be contained in exactly one of the $q+1$ conics in the pencil, this observation leads us to exactly three possible shapes for the pencil. In particular, we end up with either $q$, $q-1$ or $q-2$ proper conics. According to this, we define:

\begin{definition} The three possible shapes of a partition of the plane $PG(2,q)$ given by a pencil of disjoint conics are denoted by \emph{$q$-form} for the partition $\left\{P,C_1,...,C_q\right\}$, by \emph{$(q-1)$-form} for the partition $\left\{P,g,C_1,...,C_{q-1}\right\}$ and by \emph{$(q-2)$-form} for the partition $\left\{P,Q,g\tilde{g},C_1,...,C_{q-2}\right\}$, where $P$, $Q$ are points, $g$ is a line, $g\tilde{g}$ a pair of lines and $C_i$ proper conics, for $1 \leq i \leq q$.
\end{definition}

By analyzing these three shapes, we find the following condition for simultaneous diagonalization:
\begin{theorem} \label{DiagDisj} Two disjoint conics $V_1$ and $V_2$ in $PG(2,q)$ are simultaneously diagonalizable if and only if their pencil $\mathcal{P}(V_1,V_2)$ leads to a partition of $(q-1)$-form or $(q-2)$-form.
\end{theorem}
\begin{proof} In the first part of the proof, we will see that starting with a pencil of two diagonal polynomials yields into a $(q-1)$-form or a $(q-2)$-form. For this, let $E_1$ and $E_2$ be the polynomials defining $V_1$ and $V_2$, respectively. Consider
$$ E_1 =  x^2+b_1 y^2+c_1 z^2  \text{ and } E_2 =  x^2 + b_2 y^2 + c_2 z^2 $$
where $b_1,b_2,c_1,c_2 \neq 0$ and $b_1 \neq b_2$ or $c_1 \neq c_2$, hence we start with two distinct proper conics. Note that this is no restriction for $q \geq 5$, since in every possible form, there are at least two proper conics and by Lemma \ref{ClosednessAddition}, the pencil is independent of the choice of representatives. For $q=3$, all pencils in $1$-form can be analyzed by hand. 

Considering the pencil $\mathcal{P}(V_1,V_2)$, we see that $\mathbb{V}(E_1-E_2)$ leads to a degenerate conic. Moreover, $\mathbb{V}(E_1 - \frac{b_1}{b_2}E_2)$ and $\mathbb{V}(E_1 - \frac{c_1}{c_2}E_2)$ lead to one or two more degenerate conics, since $b_1 \neq b_2$ or $c_1 \neq c_2$. Hence, we indeed end up with a pencil in $(q-1)$-form or $(q-2)$-form.

For the other direction, we have to show that all pairs of disjoint conics with a pencil in $(q-1)$-form or $(q-2)$-form can be diagonalized simultaneously. For this, let us first look at a pencil in $(q-1)$-form. 

We start with two disjoint conics such that their pencil is of $(q-1)$-form, which means that exactly two degenerate conics occur, namely one point $P$ and one line $g$. We can now apply a collineation to the whole pencil, such that, without loss of generality, the point $P$ is given by $[1,0,0]^T$ and the line $g$ is incident with $[0,1,0]^T$ and $[0,0,1]^T$. So we can consider the pencil given by the two elements $g=\mathbb{V}(x^2)$ and $P=\mathbb{V}(y^2+cz^2+fyz)$, where $c$ and $f$ are chosen such that $y^2+cz^2+fyz$ is irreducible.

The elements of the pencil $\mathcal{P}(P,g)$ are then given by:
\begin{align*}
P &= \mathbb{V}(y^2+cz^2+fyz)\\
g &= \mathbb{V}(x^2) \\
C_0 &=\mathbb{V}(x^2+y^2+c z^2 + fyz) \\
C_1 &=\mathbb{V}(x^2+\alpha y^2+\alpha c z^2 + \alpha fyz) \\
C_2 &=\mathbb{V}(x^2+\alpha^2 y^2+\alpha^2 c z^2 + \alpha^2fyz) \\
\vdots\\
C_{q-2} &=\mathbb{V}(x^2+\alpha^{q-2} y^2+\alpha^{q-2} c z^2 + \alpha^{q-2}fyz)
\end{align*}
Now we can apply another collineation such that all these conics are in diagonal form, namely $\phi_{S^{-1}}$, for the following matrix
$$ S = \left(\begin{array}{*3{C{1.5em}}} 1 & 0 & 0 \\ 0 & 1 & -\frac{f}{2} \\ 0 & 0 & 1  \end{array}\right). $$
Indeed, we have
$$ S^{T}AS = \left(\begin{array}{*3{C{2.5em}}} 0 & 0 & 0 \\ 0 & 1 & 0 \\ 0 & 0 &  c-\frac{f^2}{4}  \end{array}\right),$$
where $A$ is the matrix representation of $P$. Note that applying $\phi_{S^{-1}}$ to $g$ gives $g$ again. Since $\phi_{S^{-1}}(E_i+E_j) = \phi_{S^{-1}}(E_i)+\phi_{S^{-1}}(E_j)$, we indeed end up with a pencil of diagonal conics, namely:
\begin{align*}
\tilde{P} &= \mathbb{V}(y^2+\tilde{c}z^2)\\
\tilde{g} &= \mathbb{V}(x^2) \\
\tilde{C}_0 &=\mathbb{V}(x^2+y^2+\tilde{c} z^2) \\
\tilde{C}_1 &=\mathbb{V}(x^2+\alpha y^2+\alpha \tilde{c} z^2) \\
\tilde{C}_2 &=\mathbb{V}(x^2+\alpha^2 y^2+\alpha^2 \tilde{c} z^2) \\
\vdots\\
\tilde{C}_{q-2} &=\mathbb{V}(x^2+\alpha^{q-2} y^2+\alpha^{q-2} \tilde{c} z^2)
\end{align*}
for $\tilde{c} = c-\frac{f^2}{4}$.

Now we look at a pencil in $(q-2)$-form. Let us start with two proper disjoint conics $C$ and $\tilde{C}$. The pencil $\mathcal{P}(C,\tilde{C})$ contains a pair of lines and two points. By applying a suitable collineation, we can assume one of the points to be $P =[1,0,0]^T$, i.e.\ $P=\mathbb{V}(y^2+cz^2+fyz)$ for $c$ and $f$ chosen such that $y^2+cz^2+fyz$ is irreducible. Moreover, by Lemma \ref{quadrilateral}, we can assume the pair of lines to be given by $g\tilde{g}$ with $g$ through $[1,1,0]^T$ and $[0,0,1]^T$ and $\tilde{g}$ through $[1,-1,0]^T$ and $[0,0,1]^T$, so we have $g\tilde{g}=\mathbb{V}(x^2-y^2)$. With the same collineations as before, we can start with diagonal conics for $P$ and $g\tilde{g}$ and obtain the following pencil:
\begin{align*}
P &=\mathbb{V}(y^2+cz^2) \\
g\tilde{g} &= \mathbb{V}(x^2-y^2)\\
C_0 &= \mathbb{V}(x^2+cz^2)\\
C_1 &= \mathbb{V}(x^2+(\alpha-1)y^2+\alpha cz^2)\\
C_2 &= \mathbb{V}(x^2+(\alpha^2-1)y^2+\alpha^2cz^2)\\
\vdots\\
C_{q-2} &= \mathbb{V}(x^2+(\alpha^{q-2}-1)y^2+\alpha^{q-2}cz^2)
\end{align*}
Note that $C_0$ corresponds to the second point in the pencil. Again, we end up with diagonal forms only.
\end{proof}

\begin{remark} It is well known that all proper conics in $PG(2,q^2)$ are projectively equivalent to $x^2+y^2+z^2=0$, see \cite{MR1612570}*{Section 5}. In particular, any conic can be diagonalized. This can also be seen as a consequence of the above result. To see this, just choose any line $g$ disjoint from $C$ and look at the pencil $\mathcal{P}(C,g)$, which is by construction of $(q-1)$-form.
\end{remark}
 
\begin{example} Let $V_1=C$ and $V_2=\tilde{C}$ be two proper disjoint conics in $PG(2,5)$ given by
$$ C = \mathbb{V}(xy+3 xz+yz) \text{ and } \tilde{C} = \mathbb{V}(x^2+y^2+z^2+xz). $$
We obtain:
\begin{align*}
V_3 &= \mathbb{V}( x^2+y^2+z^2+xy+4 xz+yz) = \left\{[1,4,1]^T \right\}\\
V_4 &= \mathbb{V}( x^2+y^2+z^2+2 xy+2 xz+2 yz ) \\
  &=\left\{[0,1,4]^T,[1,0,4]^T,[1,1,3]^T,[1,2,2]^T,[1,3,1]^T,[1,4,0]^T\right\}\\
V_5 &= \mathbb{V}( x^2+y^2+z^2+4xy+3xz+4yz) \\
V_6 &= \mathbb{V}( x^2+y^2+z^2+3xy+3yz) 
\end{align*}
Since there are four proper conics in this pencil, we can diagonalize $C$ and $\tilde{C}$ simultaneously. The first step is to perform a collineation $\phi_S$ which maps $V_3$ to $[1,0,0]^T$ and $V_4$ to the line $\mathbb{V}(x^2)$, i.e.\ we apply
$$ S^{-1} = \left(\begin{array}{*3{C{1.5em}}} 1 & 0 & 1 \\ 4 & 1 & 0 \\ 1 & 4 & 4  \end{array}\right) $$ 
to $V_3$ and $V_4$, which leads to
$$P=\mathbb{V}(y^2+3z^2+4yz)  \text{ and } g=\mathbb{V}(x^2).  $$
As the polynomial defining $P$ is not in diagonal form, we have to perform one more collineation $\phi_T$, given by the matrix
$$ T^{-1} = \left(\begin{array}{*3{C{1.5em}}} 1 & 0 & 0 \\ 0 & 1 & 1 \\ 0 & 0 & 1  \end{array}\right).$$
Combining these two steps gives the following pencil:
\begin{align*}
\tilde{P}&= \mathbb{V} (y^2+2 z^2) \\
\tilde{g}&= \mathbb{V}(x^2) \\
\tilde{C}_1 &= \mathbb{V}(x^2+y^2+2 z^2) \\
\tilde{C}_2&= \mathbb{V} (x^2+2 y^2+4 z^2) \\
\tilde{C}_3&= \mathbb{V} (x^2+4 y^2+ 3z^2) \\
\tilde{C}_4&= \mathbb{V} (x^2 + 3 y^2+ z^2)
\end{align*}
Hence, we find a collineation $\phi_M$ with
$$ M^{-1} = S^{-1} T^{-1} =\left(\begin{array}{*3{C{1.5em}}} 1 & 0 & 1 \\ 4 & 1 & 0 \\ 1 & 4 & 4  \end{array}\right)   \left(\begin{array}{*3{C{1.5em}}} 1 & 0 & 0 \\ 0 & 1 & 1 \\ 0 & 0 & 1  \end{array}\right) = \left(\begin{array}{*3{C{1.5em}}} 1 & 0 & 1 \\ 4 & 1 & 1 \\ 1 & 4 & 3  \end{array}\right)$$
such that
$$ \phi_M(C)= \left(\begin{array}{*3{C{2em}}} 1 & 0 & 0 \\ 0 & 4 & 0 \\ 0 & 0 & 3  \end{array}\right) \text{ and } \phi_M(\tilde{C})=\left(\begin{array}{*3{C{1.5em}}} 1 & 0 & 0 \\ 0 & 3 & 0 \\ 0 & 0 & 1  \end{array}\right). $$
\end{example}

Our next goal is to show that not all pairs of proper disjoint conics in $PG(2,q)$ can be diagonalized simultaneously, i.e.\ we have to show that pencils in $q$-form actually exist.

\begin{lemma} \label{NrPencils} Let $C$ be a proper conic in $PG(2,q)$. Then there are $\frac{1}{2}(2q-3q^2+q^3)$ proper conics $\tilde{C}$, such that $\mathcal{P}(C,\tilde{C})$ is in $(q-1)$-form, there are $\frac{1}{4}(-6 q + 5 q^2 + 5 q^3 - 5 q^4 + q^5)$ proper conics $\tilde{C}$ such that $\mathcal{P}(C,\tilde{C})$ is in $(q-2)$-form and there are $\frac{1}{8}( 6 q - 3 q^2 - 7 q^3 + 3 q^4 + q^5)$ proper conics $\tilde{C}$ such that $\mathcal{P}(C,\tilde{C})$ is in $q$-form.
\end{lemma}
\begin{proof} Let $C$ be any fixed proper conic. To find the number of proper conics $\tilde{C}$ such that $\mathcal{P}(C,\tilde{C})$ is in $(q-1)$-form, we can also look at the number of pencils $\mathcal{P}(C,g)$, for $g$ an external line of $C$. There are exactly $\frac{q(q-1)}{2}$ external lines of $C$. Every such pencil $\mathcal{P}(C,g)$ leads to $(q-2)$ proper conics $\tilde{C}$. Note that by Lemma \ref{ClosednessAddition}, $\tilde{C} \in \mathcal{P}(C,g)$ and $\tilde{C} \in \mathcal{P}(C,h)$ for two lines $g$ and $h$ implies $g=h$. Hence, there are $\frac{q(q-1)(q-2)}{2} = \frac{1}{2}(2q-3q^2+q^3)$ proper conics $\tilde{C}$, such that $\mathcal{P}(C,\tilde{C})$ is in $(q-1)$-form.

Similarly, to find the number of proper conics $\tilde{C}$ such that $\mathcal{P}(C,\tilde{C})$ is in $(q-2)$-form, we can also look at the number of pencils $\mathcal{P}(C,g\tilde{g})$, for $g\tilde{g}$ a pair of external lines of $C$. There are $\frac{q(q-1)}{2} (\frac{q(q-1)}{2}-1)$ choices for $g\tilde{g}$ and in every such pencil there are $q-3$ proper conics $\tilde{C}$. Hence, there are $\frac{q(q-1)}{2} (\frac{q(q-1)}{2}-1) (q-3) = \frac{1}{4}(-6 q + 5 q^2 + 5 q^3 - 5 q^4 + q^5)$ proper conics $\tilde{C}$, such that $\mathcal{P}(C,\tilde{C})$ is in $(q-2)$-form.

By Lemma \ref{NrIntersection}, there are $\frac{1}{8}(2 q - 5 q^2 + 7 q^3 - 7 q^4 + 3 q^5)$ proper conics $\tilde{C}$, which are disjoint from a given conic $C$. Subtracting the number of pencils $\mathcal{P}(C,\tilde{C})$ in $(q-1)$-form and in $(q-2)$-form, we obtain exactly $\frac{1}{8}( 6 q - 3 q^2 - 7 q^3 + 3 q^4 + q^5)$ proper conics $\tilde{C}$ such that $\mathcal{P}(C,\tilde{C})$ is in $q$-form.
\end{proof}

\begin{corollary} For every conic $C$, there exists a conic $\tilde{C}$, such that $C$ and $\tilde{C}$ cannot be diagonalized simultaneously in $PG(2,q)$.
\end{corollary}
\begin{proof} By Lemma \ref{NrPencils}, for every fixed proper conic $C$, there are $\frac{1}{8}( 6 q - 3 q^2 - 7 q^3 + 3 q^4 + q^5)$ proper conics $\tilde{C}$ such that $\mathcal{P}(C,\tilde{C})$ is in $q$-form, a number which is strictly greater than zero, for all $q \geq 3$. By Theorem \ref{DiagDisj}, these pairs of conics $C$ and $\tilde{C}$ cannot be diagonalized simultaneously.
\end{proof}

\begin{remark} In the real projective plane, all pairs of disjoint conics can be diagonalized simultaneously (see \cite{MR0094690}). Note that in this case, two disjoint conics $C$ and $\tilde{C}$ have the property that either all points of $C$ are external points of $\tilde{C}$ or all points of $C$ are internal points of $\tilde{C}$, as they do not intersect. In $PG(2,q)$ however, the two conics can be disjoint without having this property, i.e.\ $C$ can contain external points of $\tilde{C}$ as well as internal points. We will show in the remaining part of this subsection, that if all conics in a pencil are disjoint and for each pair of proper conics in this pencil, one conic consists of only external points or only internal points of the other one, simultaneous diagonalization is still possible. 
\end{remark}

We are therefore interested in the following property of conic pairs:

\begin{definition} Let $C$ and $\tilde{C}$ be two proper conics in $PG(2,q)$. We say that $C$ and $\tilde{C}$ are in {\em nested position} if every point of $C$ is an external point of $\tilde{C}$ or every point of $C$ is internal point of $\tilde{C}$ and vice versa, i.e.\ every point of $\tilde{C}$ is an external point of $C$ or every point of $\tilde{C}$ is internal point of $C$.
\end{definition}

\begin{remark} Note that if $C$ consists of internal points of $\tilde{C}$, that does not imply that $\tilde{C}$ consists of external points of $C$. In particular, in the finite projective plane $PG(2,q)$ it is possible that two conics $C$ and $\tilde{C}$ are such that $C$ consists of internal points of $\tilde{C}$ and $\tilde{C}$ consists of internal points of $C$ as well. Because of that, we need to write down both conditions in the definition above.
\end{remark}

The interested reader can refer to \cite{MR2852921} for information 
about the maximal number of external points on $C$ with respect to another given conic $\tilde{C}$ as well as the number of conics $C$ and $\tilde{C}$ which lie in nested position. Moreover, about the construction of conics consisting only of external points of a given conic, one can have a look at \cite{MR2854629}.

We will use the following well-known result about nested conics:

\begin{lemma} Let $C$ and $\tilde{C}$ be in nested position. Then $C$ and $\tilde{C}$ are disjoint and have no tangents in common.
\end{lemma}
\begin{proof} The points of $C$ are either all external or all internal to $\bar{C}$. Thus, $C$ does not contain a point of $\tilde{C}$ and hence, they are disjoint. For the second property, we have to distinguish two cases. We start with $C$ consisting of internal points of $\tilde{C}$ only. In this case, no point of $C$ is incident with a tangent of $\tilde{C}$, i.e.\ the two conics have no tangents in common. 
For $C$ consisting of external points of $\tilde{C}$ only, we know that through every point of $C$ there are exactly two tangents of $\tilde{C}$ . There are $q+1$ points on $C$ and there are $q+1$ tangents of $\tilde{C}$. As every point of $C$ is incident with two tangents of $\tilde{C}$, also every tangent of $\tilde{C}$ need to be incident with two points of $C$, hence every tangent of $\tilde{C}$ is a secant of $C$.
\end{proof}

\begin{remark} \label{NestedInvariant} As a direct consequence of the definitions, we know that the property of two proper conics lying in nested position is invariant under collineations.
\end{remark}

\begin{theorem} The only pencil of two proper disjoint conics $C$ and $\tilde{C}$ in $PG(2,q)$ for $q$ odd, where all pairs of proper conics lie in nested position, is a pencil $\mathcal{P}(C,\tilde{C})$ in $(q-1)$-form. Moreover, a pencil $\mathcal{P}(C,\tilde{C})$ in $(q-1)$-form has the property that all pairs of proper conics lie in nested position.
\end{theorem}
\begin{proof} The first statement can be proven by combinatorial considerations only, namely by excluding that such a pencil can be in $q$-form or $(q-2)$-form. There are $\frac{q(q+1)}{2}$ external points for any fixed proper conic $C$. Assume that all pairs of proper conics in a pencil in $q$-form are in nested position. For this, we have to consider two cases, namely the unique point $P$ in the pencil being an external point or an internal point of $C$. Let $P$ be an external point of $C$. By our assumption, this gives $k(q+1)+1$ external points of $C$, for an integer $k$, depending on the number of proper conics in the pencil consisting of external points of $C$ only. This means
$$ k(q+1)+1 = \frac{q(q+1)}{2} $$
and therefore $k = \frac{q}{2} - \frac{1}{q+1}$, which is not an integer, because $q$ is odd. Hence, this case is not possible. 

Now, let $P$ be an internal point of $C$. Again, by assumption, this would imply $k(q+1) = \frac{q(q+1)}{2}$ for an integer $k$, which is not possible for $q$ odd. Because of this, not all pairs of proper conics in a pencil of $q$-form can be in nested position.

To exclude that such a pencil can be in $(q-2)$-form, assume again that all pairs of conics lie in nested position. Here, we have to consider different cases of how many points of $g\tilde{g}, P,\tilde{P}$ are external points of $C$. The lines $g$ and $\tilde{g}$ intersect in exactly one point, say $R$. If $R$ is an external point of $C$, there are two tangents of $C$ through $R$. Since every two lines intersect, every other of the remaining $q-1$ tangents of $C$ must intersect both $g$ and $\tilde{g}$ in different points, which gives $1+(q-1)=q$ external points on $g\tilde{g}$. For the case that $R$ is an internal point of $C$, by the same argument, we obtain $q+1$ external points on $g\tilde{g}$. Moreover, the two points $P$ and $\tilde{P}$ can be either external points or internal points of $C$, which gives us in total $q,q+1,q+2$ or $q+3$ external points not on any proper conic of the pencil. Hence, we obtain the condition
$$k(q+1)+i=\frac{q(q+1)}{2}$$
for some integer $k$ and $i \in \left\{-1,0,1,2\right\}$, which is not possible for $q$ odd and $q \geq 5$. Hence, not all pairs of proper conics in a pencil of $(q-2)$-form can be in nested position. Note that the statement is trivially true for $q=3$ as well, since a pencil in $q$-form for $q=3$ consists of only four elements, which are two points, one pair of lines and only one proper conic. Hence, we do not even have two proper conics in this case.

To show the second statement, remember that by Lemma \ref{NestedInvariant}, the property of lying in nested position is invariant under collineations. By assumption, the pencil is in $(q-1)$-form and hence can be transformed into a pencil of diagonal conics, as shown in Theorem \ref{DiagDisj}. Therefore, it is enough to show that in the following pencil, where $c$ is chosen to be a non-square, all pairs of proper conics lie in nested position:
\begin{align*}
P &= \mathbb{V}(y^2+c z^2)\\
g &= \mathbb{V}(x^2) \\
C_0 &=\mathbb{V}(x^2+y^2+c z^2) \\
C_1 &=\mathbb{V}(x^2+\alpha y^2+\alpha c z^2) \\
C_2 &=\mathbb{V}(x^2+\alpha^2 y^2+\alpha^2 c z^2) \\
\vdots\\
C_{q-2} &=\mathbb{V}(x^2+\alpha^{q-2} y^2+\alpha^{q-2} c z^2)
\end{align*}

So, let $C_i$ and $C_j$ be given by
$$ C_i=\mathbb{V}(x^2+\alpha^i y^2+\alpha^i c z^2) \text{ and } C_j=\mathbb{V}(x^2+\alpha^j y^2+\alpha^j c z^2).$$
Since all points of $PG(2,q)$ with a zero x-coordinate lie on the line $g= \mathbb{V}(x^2)$, every point $P$ of $C_i$ can be written in the form $P=[1,p_2,p_3]^T$ and in particular, we have $p_2^2 =
-\alpha^{-i}-c p_3^2$.  The  conics $C_i$ and $C_j$  lie in nested position
if either for all such points $P$ of $C_i$, $(A_j P)^T$
is a secant of $C_j$ or this happens for no such point $P$, where $A_j$ is the matrix representation of $C_j$, i.e.\
$$ A_j = \left(\begin{array}{*3{C{1.5em}}} 1 & 0 & 0 \\ 0 & \alpha^j & 0 \\ 0 & 0 & c \alpha^j \end{array}\right). $$
We know that $[1, \alpha^i p_2,c \alpha^i p_3]$ is a secant of $C_j$ if and only if $\alpha^{-j} x + p_2 y + c p_3 z=0$ has two solutions for points $[x,y,z]^T$ in $C_j$. Since the x-coordinate of such points is not zero, we are looking for solutions of the form $[1,\pm \sqrt{-\alpha^{-j}-cz^2},z]^T$. Hence, we have to find solutions of
$$z^2 - 2 \alpha^{i-j} p_3 z + \alpha^{-j} c^{-1} + \alpha^{i-j} p_3^2- \alpha^{i-2j} c^{-1}=0.$$
This quadratic equation is solvable for $z$ if and only if its discriminant is a
square in $GF(q)$, i.e.\ if and only if $\alpha^{2i-2j}p_3^2 - \alpha^{-j}c^{-1}-p_3^2 \alpha^{i-j}+\alpha^{i-2j}c^{-1}$ is a square in $GF(q)$, which is the same as $(\alpha^i p_3^2+c^{-1})(\alpha^i-\alpha^j)$ being a square in $GF(q)$. Since $[1,p_2,p_3]^T$ lies on $C_i$, we know $\alpha^i p_3^2+c^{-1}=-c^{-1}\alpha^i p_2^2$. The condition is therefore that $(-\alpha^i c^{-1})(\alpha^i-\alpha^j)$ needs to be a square, which is independent of the point $P$, and hence $C$ and $\tilde{C}$ are nested.
\end{proof}

\subsection{The non-disjoint case}

Note that two non-disjoint proper conics can intersect in exactly one, two, three or four points. If they intersect in more than four points, they are actually the same, since any proper conic is uniquely defined by five of its points. To study the question whether or not two proper non-disjoint conics can be diagonalized simultaneously, we look at the following results, each concerned with a different number of intersection points. For all cases, we assume simultaneous diagonalization for the two proper conics we start with. Hence, in all proofs, we have to consider the following pencil:
\begin{align*}
C_1 &= \mathbb{V}( x^2+by^2+cz^2)\\
C_2 &= \mathbb{V}( x^2+\tilde{b}y^2+\tilde{c}z^2) \\
V_0 &= \mathbb{V}( (1+1)x^2+(b+\tilde{b})y^2 + (c+\tilde{c})z^2) \\
V_1 &= \mathbb{V}( (1+\alpha)x^2+(b+\alpha\tilde{b})y^2 + (c+\alpha\tilde{c})z^2 ) \\
V_2 &= \mathbb{V}( (1+\alpha^2)x^2+(b+\alpha^2\tilde{b})y^2 + (c+\alpha^2\tilde{c})z^2) \\
\vdots\\
V_{\frac{q-1}{2}} &= \mathbb{V}((b-\tilde{b})y^2 + (c-\tilde{c})z^2)\\
\vdots\\
V_{q-2} &= \mathbb{V}((1+\alpha^{q-2})x^2+(b+\alpha^{q-2}\tilde{b})y^2 + (c+\alpha^{q-2}\tilde{c})z^2) \\
\end{align*}
for $b, \tilde{b}, c, \tilde{c} \neq 0$ and $b \neq \tilde{b}$ or $c \neq \tilde{c}$.
Note that $V_{\frac{q-1}{2}}$ is a degenerate conic, since $\alpha^{\frac{q-1}{2}}=-1$. Since $b +k \tilde{b}$ and $c +k \tilde{c}$ runs through all elements of $GF(q)$ for $k=0,...,q-1$, there are exactly two or three degenerate conics in this pencil.

\begin{remark} Note that in a pencil of conics which contains more than three degenerate conics, all conics are degenerate. This can for example be seen by setting $b=0$ and $\tilde{c}=0$ in the above pencil, i.e.\ we start with two degenerate conics. Of course, $V_{\frac{q-1}{2}}$ is still a degenerate conic. To obtain more than three degenerate conics, we necessarily need $c=0$ or $\tilde{b}=0$ as well, hence all conics in the pencil are degenerate.
\end{remark}

\begin{theorem} No two proper conics $C_1$ and $C_2$ in $PG(2,q)$, which intersect in exactly one point, can be diagonalized simultaneously.
\end{theorem}
\begin{proof} Let $C_1$ and $C_2$ intersect in exactly one point, say $P$, and assume that $C_1$ and $C_2$ can be diagonalized simultaneously. We know that all elements in the pencil $\mathcal{P}(C_1,C_2)$ must contain $P$ as well and we have to distinguish the cases of $P$ itself occurring as an element in the pencil or not. 

Let us start by assuming that $P$ itself is a conic of the pencil. Since we only have $q+1$ elements in the pencil and every point of the plane $PG(2,q)$ needs to occur in one of these, there has to be a degenerate conic corresponding to a pair of lines as well. Indeed, this gives $1+2q+(q-1)q=q^2+q+1$ different points in the $q+1$ conics of the pencil, as no point except $P$ occurs in more than one element of the pencil. As there are equally many points on a proper conic and on a line, the remaining $q-1$ conics can correspond to lines or proper conics. There are two or three degenerate conics in the pencil of $C_1$ and $C_2$, hence we have to distinguish further.

First, assume that the pencil is of the form $\{P,g\tilde{g},C_1,...,C_{q-1}\}$. In this case, we have exactly two degenerate conics and two of the elements of the pencil above must correspond to a point and a pair of lines, respectively. This means $b-\tilde{b}\neq 0$ and $c-\tilde{c}\neq 0$ as otherwise, $V_{\frac{q-1}{2}}$ would correspond to a line. As the degenerate conics are exactly two, there exists a $k$ such that $b+\alpha^k \tilde{b}=0$ and $c+\alpha^k \tilde{c}=0$. But then, the conic $V_k$ is a line, which is a contradiction.

Now, let the pencil be of the form $\{P,g\tilde{g},g',C_1,...,C_{q-2}\}$. Here, we have three degenerate conics. Assume that $V_{\frac{q-1}{2}}$ is a line. Since each zero of $x^2+ty^2$ corresponds to a point or a pair of lines, depending on $t\neq 0$, we need $b=\tilde{b}$ or $c=\tilde{c}$ to obtain the line $\mathbb{V}(z^2)$ or $\mathbb{V}(y^2)$. In both cases, there can only be one more degenerate conic, which is a contradiction. So $V_{\frac{q-1}{2}}$ must correspond to the point or the pair of lines. In both cases, we have $b\neq\tilde{b}$ and $c\neq\tilde{c}$. Since we need three degenerate conics, there exists a $k$ such that $b+\alpha^k\tilde{b}=0$ and a $k'$ with $c+\alpha^{k'}\tilde{c}=0$. For $k\neq k'$, the conics $V_k$ and $V_{k'}$ both correspond to a point or a pair of lines. As $k=k'$ leads to only one more degenerate conic corresponding to a line, we have a contradiction.

Now assume that $P$ does not occur as an element of the pencil. In this case, assume we have the pencil $\{g_1,g_2,C_1,...,C_{q-1}\}$ or $\{g_1,g_2,g_3,C_1,...,C_{q-2}\}$. This means that $V_{\frac{q-1}{2}}$ must correspond to a line, i.e.\ $b=\tilde{b}$ or $c=\tilde{c}$. This gives only one more degenerate conic, which corresponds to a point or a pair of lines, hence we cannot produce the pencils we assumed.
\end{proof}

\begin{remark} \label{Remark1} Note that the same result is true in the real projective plane as well. A shorter proof for that is to look at two proper conics in diagonal form, namely $C= \mathbb{V}(x^2 +by^2+cz^2)$ and $\tilde{C}=\mathbb{V}(x^2 + \tilde{b}y^2+\tilde{c}z^2)$. Note that if $P=[x,y,z]^T$ lies on $C$ and $\tilde{C}$, then so do $[x,-y,z]^T$, $[x,y,-z]^T$ and $[x,-y,-z]^T$. Since $[1,0,0]^T$, $[0,1,0]^T$ and $[0,0,1]^T$ do not lie on a proper conic of this pencil, we obtain exactly two or four intersection points. This argument holds for the real projective plane as well as for $PG(2,q)$. For this reason, we expect simultaneous diagonalization to fail for conics with exactly one or exactly three common points. The proof above shows that these results can be obtained by considering the pencil of conics as well, although the proof is much longer than the simple argument just mentioned.
\end{remark}

\begin{lemma} Two proper conics $C_1$ and $C_2$ which intersect in exactly two points can be diagonalized simultaneously if and only if their pencil is of the form
$$ \{g_1\tilde{g_1},g_2,C_1,C_2,...,C_{q-1}\}.$$
\end{lemma}
\begin{proof} Let $C_1$ and $C_2$ intersect in exactly two points $P$ and $Q$. None of the elements of the pencil can correspond to a point, since both $P$ and $Q$ lie in all elements of the pencil. Moreover, since every point of the plane except $P$ and $Q$ needs to be in exactly one element of the pencil, we need one pair of lines, since we need $(2q+1)+(q-1)q=q^2+q+1$ points in total. As there is exactly one line through $P$ and $Q$, we only have two possible forms for the pencil, namely $ \{g_1\tilde{g_1},g_2,C_1,C_2,...,C_{q-1}\}$ or $ \{g_1\tilde{g_1},C_1,C_2,...,C_q\}$. The second pencil has only one degenerate conic, so simultaneous diagonalization is not possible in this case, which shows one direction of our claim. 

For the other direction, let the pencil be of the form $\{ g_1\tilde{g_1},g_2,C_1,C_2,...,C_{q-1}\}.$ Note that for two diagonal conics, if they intersect in a point $[1,y,z]^T$ they intersect as well in the points $[1,-y,z]^T$, $[1,y,-z]^T$ and $[1,-y,-z]^T$, which gives potentially four intersection points. Hence, if the two conics we started with are diagonalizable simultaneously and intersect in exactly two points, they have to intersect in a point with a zero coordinate, e.g. $[1,\alpha^k,0]^T$, which gives only one more intersection point. As we can transform any three non-collinear points to any other three non-collinear points, we search for a transformation $\phi_S$ such that $\phi_S(P)=[1,\alpha^k,0]^T$, $\phi_S(Q)=[1,-\alpha^k,0]^T$ and $\phi_S(g_1\cap\tilde{g_1})=[0,0,1]^T$. Note that the intersection of the pair of lines, i.e.\ $g_1 \cap \tilde{g_1}$ cannot be $P$ or $Q$, since then $g_1$ or $\tilde{g_1}$ would be the line through $P$ and $Q$, namely $g_2$, and the elements $g_1\tilde{g_1}$ and $g_2$ would have $q+1$ common points. So, we indeed obtain, without loss of generality,
$g_2=\mathbb{V}(z^2)$ and $g_1\tilde{g_1}=\mathbb{V}(\alpha^{2k} x^2-y^2)$. All elements of the pencil intersect in exactly two points, namely $[1,\alpha^k,0]^T$ and $[1,-\alpha^k,0]^T$ and hence the two conics we started with are mapped into two diagonal conics, which intersect in exactly those two points as well.
\end{proof}

\begin{lemma} Two proper conics $C_1$ and $C_2$ which intersect in exactly three points, cannot be diagonalized simultaneously.
\end{lemma}
\begin{proof} Let $C_1$ and $C_2$ intersect in the points $P$, $Q$ and $R$. Note that none of the conics in the pencil corresponds to a point. Moreover, none of the conics corresponds to a line, as $P$, $Q$ and $R$ are not collinear by the definition of a proper conic. Hence, the only degenerate conics which occur correspond to pairs of lines. The question therefore is how many such pairs there are and if there is only one possible pencil. For this, note that on a pair of lines, there are always $2q+1$ points. As all conics in question intersect in the same three points, we have to be careful not to count any point too often. So, we have to solve the following equation:
$$ 3 + i (2q-2) + (q+1-i)(q-2) = q^2+q+1,\ 0 \leq i \leq 3 $$
which gives immediately $i=2$, so there are always two degenerate conics which correspond to a pair of lines. Hence, the pencil of two proper conics which intersect in three points is always of the form $\{g_1\tilde{g_1},g_2\tilde{g_2},C_1,...,C_{q-1}\}$. 

Therefore, the conic $V_{\frac{q-1}{2}}$ corresponds to a pair of lines, which gives $b\neq\tilde{b}$ and $c\neq\tilde{c}$. There is only one more degenerate conic, hence there exists a $k$ such that $b+\alpha^k\tilde{b}=0$ and $c+\alpha^k\tilde{c}=0$. But then $V_k$ corresponds to a line, which is a contradiction.
\end{proof}

\begin{remark} We already mentioned in Remark \ref{Remark1} that this happens in the real projective plane as well.
\end{remark}

\begin{lemma} Two proper conics $C_1$ and $C_2$ which intersect in exactly four points can always be diagonalized simultaneously.
\end{lemma}
\begin{proof} In the pencil of two proper conics $C_1$ and $C_2$ which intersect in exactly four points, the only possible degenerate conics correspond to pairs of lines by the same argument as before. Again, we have to think about the number $i$ of pairs of lines needed to ensure that every point of the plane is in an element of this pencil. For this, we have to solve
$$ 4 + i (2q-3) + (q+1-i)(q-3) = q^2+q+1 $$
which gives us $i=3$. Hence, the pencil of two proper conics which intersect in exactly four points is always given by $\{g_1\tilde{g_1},g_2\tilde{g_2},g_3\tilde{g_3},C_1,...,C_{q-2}\}$. These four intersection points are such that no three are on one line and hence by Lemma \ref{quadrilateral}, we can choose those points as:
$$ P =[1,1,1]^T,\ Q=[1,-1,1]^T,\ R=[1,1,-1]^T,\ S=[1,-1,-1]^T $$
There are exactly three different pairs of lines through these four points $P$, $Q$, $R$ and $S$, given by:
$$ g_1\tilde{g_1} = \mathbb{V}( y^2-z^2),\  g_2\tilde{g_2}=\mathbb{V}( x^2-y^2),\  g_3\tilde{g_3}= \mathbb{V}( x^2-z^2)$$
This leads to the following pencil:
\begin{align*}
g_1\tilde{g_1}&= \mathbb{V}( y^2-z^2 ) \\
g_2\tilde{g_2}&= \mathbb{V}( x^2-y^2) \\
g_3\tilde{g_3} &= \mathbb{V}(x^2-z^2) \\
C_1&= \mathbb{V}( x^2+(\alpha-1)y^2-\alpha z^2) \\
C_2&= \mathbb{V}( x^2+(\alpha^2-1)y^2-\alpha^2 z^2 ) \\ 
\vdots\\
C_{q-2}&= \mathbb{V}( x^2+(\alpha^{q-2}-1)y^2-\alpha^{q-2} z^2 ) \\
\end{align*}
Therefore, the two conics we started with are in diagonal form as well.
\end{proof}

\begin{remark} In the real projective plane, two conics which intersect in exactly four points can always be diagonalized simultaneously as well. This can for example be shown by proving the existence of a triangle which is self-polar with respect to both conics (see for example \cite{MR2305055}).
\end{remark}

\section{Summary}

To finish the discussion about simultaneous diagonalization of two symmetric $3 \times 3$ matrices over $GF(q)$ using pencils of conics, we summarize our results in the following table.

\begin{center}
	\begin{tabular} {c|c|c}
	\textbf{Case} & \textbf{Pencil} & \textbf{diagonalizable?} \\
	\hline
	$|C_1\cap C_2|=0$ & $\{P,g,C_1,...,C_{q-1}\}$ & yes\\
	                  & $\{P,\tilde{P},g\tilde{g},C_1,...,C_{q-2}\}$ & yes\\
										& $\{P,C_1,...,C_q\}$ & no\\
	\hline
	$|C_1\cap C_2|=1$ & $\{P,g\tilde{g},C_1,C_2,C_3/g_3,...,C_{q-1}/g_{q-1}\}$ & no\\
	                  & $\{C_1,C_2,C_3/g_3,...,C_{q+1}/g_{q+1}\}$ & no\\
	\hline
	$|C_1\cap C_2|=2$ & $\{g\tilde{g},g',C_1,...,C_{q-1}\}$ & yes\\
	                  & $\{g\tilde{g},C_1,...,C_q\}$ & no\\
	\hline
	$|C_1\cap C_2|=3$ & $\{g_1\tilde{g_1},g_2\tilde{g_2},C_1,...,C_{q-1}\}$ & no\\
	\hline
	$|C_1\cap C_2|=4 $& $\{g_1\tilde{g_1},g_2\tilde{g_2},g_3\tilde{g_3},C_1,...,C_{q-2}\}$ & yes\\
	 \end{tabular}
\end{center}

\section{Acknowledgements}
I would like to thank Norbert Hungerb\"uhler, G\'abor Korchm\'aros and Lukas Parapatits for helpful comments on this paper.



\begin{thebibliography}{1}

\bib{MR2852921}{article}{
   author={Abatangelo, V.},
   author={Fisher, J. C.},
   author={Korchm{\'a}ros, G.},
   author={Larato, B.},
   title={On the mutual position of two irreducible conics in ${\rm
   PG}(2,q)$, $q$ odd},
   journal={Adv. Geom.},
   volume={11},
   date={2011},
   number={4},
   pages={603--614},
}

\bib{MR1612570}{book}{
   author={Hirschfeld, J. W. P.},
   title={Projective geometries over finite fields},
   series={Oxford Mathematical Monographs},
   edition={2},
   publisher={The Clarendon Press, Oxford University Press, New York},
   date={1998},
   pages={xiv+555},
}
	

\bib{MR818902}{article}{
   author={Hong, Yoo Pyo},
   author={Horn, Roger A.},
   author={Johnson, Charles R.},
   title={On the reduction of pairs of Hermitian or symmetric matrices to
   diagonal form by congruence},
   journal={Linear Algebra Appl.},
   volume={73},
   date={1986},
   pages={213--226},
}

\bib{MR1429394}{book}{
   author={Lidl, Rudolf},
   author={Niederreiter, Harald},
   title={Finite fields},
   series={Encyclopedia of Mathematics and its Applications},
   volume={20},
   edition={2},
   note={With a foreword by P. M.\ Cohn},
   publisher={Cambridge University Press, Cambridge},
   date={1997},
   pages={xiv+755},
}

\bib{MR2854629}{article}{
   author={Madison, Adonus L.},
   author={Wu, Junhua},
   title={On binary codes from conics in ${\rm PG}(2,q)$},
   journal={European J. Combin.},
   volume={33},
   date={2012},
   number={1},
   pages={33--48},
}
	

\bib{MR0094690}{article}{
   author={Pesonen, Erkki},
   title={\"Uber die Spektraldarstellung quadratischer Formen in linearen
   R\"aumen mit indefiniter Metrik},
   language={German},
   journal={Ann. Acad. Sci. Fenn. Ser. A. I. no.},
   volume={227},
   date={1956},
   pages={31},
}

\bib{MR2305055}{book}{
   author={Ram{\'{\i}}rez Galarza, Ana Irene},
   author={Seade, Jos{\'e}},
   title={Introduction to classical geometries},
   note={Translated from the 2002 Spanish original},
   publisher={Birkh\"auser Verlag, Basel},
   date={2007},
   pages={x+219},
}
\end{thebibliography}
\end{document}